\documentclass[a4paper,12pt,leqno]{amsart}
\usepackage[a4paper, top=3cm, bottom=3cm, left=3cm, right=3cm]{geometry}
\usepackage{amsmath,amsfonts,amsthm,amssymb,dsfont}
\usepackage[alphabetic]{amsrefs}
\usepackage[OT4]{fontenc}
\usepackage{enumerate}
\usepackage{mathrsfs}
\usepackage{enumerate, xspace}
\usepackage[pdftex]{graphicx}
\usepackage{color}
\usepackage{hyperref}

\long\def\symbolfootnote[#1]#2{\begingroup
	\def\thefootnote{\fnsymbol{footnote}}\footnote[#1]{#2}\endgroup}

\newtheorem{theorem}{Theorem}[section]
\newtheorem{lemma}[theorem]{Lemma}
\newtheorem{thm}[theorem]{Theorem}

\newtheorem{prop}[theorem]{Proposition}
\newtheorem{cor}[theorem]{Corollary}

\newtheorem*{mainthm}{Main Theorem}
\newtheorem*{maincor}{Main Corollary}

\theoremstyle{definition}

\newtheorem{defin}[theorem]{Definition}

\newcommand{\id}{\mathrm{id}}

\newcommand{\W}{\mathcal{W}}
\newcommand{\U}{\mathcal{U}}

\begin{document}
	
	\title{Coxeter groups are biautomatic}

	\author[D.~Osajda]{Damian Osajda$^{\dag}$}
	\address{Instytut Matematyczny,
		Uniwersytet Wroc\l awski\\
		pl.\ Grun\-wal\-dzki 2/4,
		50--384 Wroc\-{\l}aw, Poland}
	\address{Institute of Mathematics, Polish Academy of Sciences\\
		\'Sniadeckich 8, 00-656 War\-sza\-wa, Poland}
	\email{dosaj@math.uni.wroc.pl}
	\thanks{$\dag \ddag$ Partially supported by (Polish) Narodowe Centrum Nauki, UMO-2018/30/M/ST1/00668.}
	
	\author[P.~Przytycki]{Piotr Przytycki$^{\ddag}$}
	
	\address{
		Department of Mathematics and Statistics,
		McGill University,
		Burnside Hall,
		805 Sherbrooke Street West,
		Montreal, QC,
		H3A 0B9, Canada}
	
	\email{piotr.przytycki@mcgill.ca}
	
	\thanks{$\ddag$ Partially supported by NSERC, AMS, and LabEx CARMIN, ANR-10-LABX-59-01}
	
	\maketitle
	
	\begin{abstract}
		\noindent
		We prove that Coxeter groups are biautomatic. From our construction of the biautomatic structure it follows that uniform lattices in isometry groups of buildings are biautomatic.
	\end{abstract}
	
	\section{Introduction}
	\label{sec:introd}
	
	Coxeter groups  were introduced in 1934 as abstractions of reflection groups \cite{Coxeter_1934}. 
	They play a fundamental role in, among others, the theory of Lie groups and algebras, the representation theory, the geometry of Riemannian symmetric spaces, the topology of aspherical manifolds. Consequently, Coxeter groups are important in other areas of science, e.g.\  physics, chemistry, and biology. They are foundational objects for buildings --- highly symmetric spaces having deep connections with algebraic groups.  On the other hand, multiple existing ways of constructing them, make Coxeter groups a source of numerous important, often very exotic, examples of groups. 
	Being studied thoroughly over decades, many important algebraic, geometric, and algorithmic properties of Coxeter groups have been established.
	Among few most important basic open problems concerning Coxeter groups, there has been the question of biautomaticity.
	
	The notion of biautomaticity was introduced in the classical book by  Epstein--Cannon--Holt--Levy--Paterson--Thurston \cite{Epstein_et_al_1992} as a very powerful means of understanding a group.
	Having biautomaticity established for a finitely generated group, very roughly speaking, we know how to move, using the generators, between any two given elements of the group. Moreover, the resulting paths are determined by a finite state automaton,
	and are stable in the sense that changing slightly the endpoints does not perturb the paths too much. Such a property should be thought of as a strong form of controlling the structure of the group.
	
	\begin{mainthm}
		Every Coxeter group is biautomatic.
	\end{mainthm}
	
	Many partial results in this direction have been obtained in the past. 
	Davis--Shapiro \cite{Davis-Shapiro_1991} showed a conjecturally weaker feature of all Coxeter groups --- the automaticity, under the assumption of the Parallel Wall Theorem, and showed that their language does not provide a biautomatic structure.
	Brink--Howlett \cite{Brink-Howlett_1993} proved the Parallel Wall Theorem and hence established the automaticity of all Coxeter groups using the same language as \cite{Davis-Shapiro_1991}.  Biautomaticity  has been established for a few subclasses of Coxeter groups in: \cite{Epstein_et_al_1992} (Euclidean and Gromov hyperbolic), \cite{Niblo-Reeves_1998,Niblo_Reeves_2003} (right-angled),
	\cite{Bahls_2006,Caprace_Mulherr_2005} (no Euclidean reflection triangles), \cite{Caprace_2009} (relatively hyperbolic), \cite{Munro_et_al_2022} ($2$-dimensional).
	
	Furthermore, there is an intensive research effort in deeper understanding languages in Coxeter groups, for example the Davis--Shapiro--Brink--Howlett language, see e.g.\  \cite{Casselman_1994,Dyer-Hohlweg_2016,Yau_2021,Parkinson-Yau_2022} and references therein. This is primarily inspired by computations in representation theory, and leads to theoretical results concerning algorithmic aspects of the languages, as well as to explicit computations and software implementations of corresponding algorithms.
	
	For our proof of the Main Theorem we introduce a new geodesic language: the `voracious' language $\mathcal V$. Besides providing a biautomatic structure, it has other interesting features compared to the previously considered languages, e.g.\ to the aforementioned Davis--Shapiro--Brink--Howlett language. In particular, $\mathcal V$ is preserved by the automorphisms of a Coxeter group preserving its given generating set. An immediate consequence of this property, together with a result by \'Swi{\k a}tkowski \cite[Thm~6.7]{Swiatkowski_2006} on geodesic languages for Coxeter groups, is the following.   
	
	\begin{maincor}
		Uniform lattices in isometry groups of buildings are biautomatic.
	\end{maincor}
	
	A uniform lattice here means a group acting properly and cocompactly on the Davis realisation of a building associated to a Coxeter group. Previously, the biautomaticity of such lattices has been shown in few particular cases in: \cite{Epstein_et_al_1992,Cartwright-Shapiro_1995} (Gromov hyperbolic and some Euclidean), \cite{Niblo-Reeves_1998,Davis_1998} (right-angled), \cite{Gersten-Short_1990,Gersten-Short_1991,Noskov_2000,Swiatkowski_2006} (some Euclidean cases), \cite{Munro_et_al_2022} ($2$-dimensional).
\medskip	
	
\noindent \textbf{Organisation.} 
In Section~\ref{sec:intro} we recall the notions of a Coxeter group and a biautomatic structure, and we define the voracious projection and language $\mathcal V$ used to prove the Main Theorem.
In Section~\ref{sec:well defined}, we show that the voracious projection is well defined. In Section~\ref{sec:bound}, we prove that the
distance between any element of $W$ and its voracious projection is bounded above by a constant depending only on $W$. We verify parts~(ii) and~(iii)
of the definition of biautomaticity in Section~\ref{sec:fellow}. In Section~\ref{sec:regularity}, we prove the regularity of $\mathcal V$.

\medskip

\noindent \textbf{Acknowledgement.} We thank Adrien Abgrall, Pierre-Emmanuel Caprace, Chris Hruska, Jingyin Huang, and Zachary Munro for useful discussions. This paper was written during our stay at the Institut Henri Poincar\'e in Paris, which we thank for the hospitality.

\section{Preliminaries}
\label{sec:intro}

We follow the notation adopted in \cite{Munro_et_al_2022}. A \emph{Coxeter group} $W$ of \emph{rank} $k$ is a group generated by a finite set $S$ of size $k$ subject only to relations $s^2=1$
for $s\in S$ and $(st)^{m_{st}}=1$ for $s\neq t\in S$, where $m_{st}=m_{ts}\in \{2,3,\ldots,\infty\}$. Here the convention is that $m_{st}=\infty$
means that we do not impose a relation between $s$ and~$t$.

Consider an arbitrary group $G$ with a finite symmetric generating set
$S$. For $g\in G$, let $\ell(g)$ denote
the \emph{word length} of $g$, that is, the minimal number $n$
such that $g=s_1\cdots s_n$ with $s_i\in S$ for $i=1,\ldots,
n$. Let $S^*$ denote the set of all words over $S$.
If $v\in S^*$ is a word of length $n$, then by $v(i)$ we
denote the prefix of $v$ of length $i$ for $i=1,\ldots, n-1$,
and the word $v$ itself for $i\geq n$. For $1\leq i\leq j\leq
n,$ by $v(i,j)$ we denote the subword of $v(j)$ obtained by
removing $v(i-1)$. For a word $v\in S^*$, by $\ell(v)$ we denote
the word length of the group element that $v$ represents.

We say that $G$ is \emph{biautomatic} if there exists a regular language $\mathcal L\subseteq S^*$ (see Section~\ref{sec:regularity} for the
definition of regularity) and constants $C,C'$ satisfying the following conditions.
\begin{enumerate}[(i)]
\item For each $g\in G$, there is a word in $\mathcal L$
    representing $g$.
\item For each $s\in S$ and $g,g'\in G$ with $g'=gs,$ and each $v,v'\in \mathcal L$
    representing $g,g'$, for
    all $i\geq 1$ we have $\ell\big(v(i)^{-1}v'(i)\big)\leq C$.
\item For each $s\in S$ and $g,g'\in G$ with $g'=sg,$ and each $v,v'\in \mathcal L$ representing $g,g'$, for all $i\geq 1$ we have
    $\ell\big(v(i)^{-1}s^{-1}v'(i)\big)\leq C'$.
\end{enumerate}

This definition agrees with the characterisation of biautomaticity in  \cite[Lem~2.5.5]{Epstein_et_al_1992}, which is equivalent to the original definition of biautomaticity if in condition (i) the set of words in $\mathcal L$ representing each $g\in G$ is finite \cite[Thm~6]{Amrhein_2021}. Conditions (ii) and (iii) are called the `fellow traveller property'.

To define the voracious language, we need the following. By $X^1$ we denote the \emph{Cayley graph} of $W$, that is, the graph with vertex set $X^0=W$
and with edges (of length $1$) joining each $g\in W$ with $gs$, for $s\in S$. We
consider the action of $W$ on $X^0=W$ by left multiplication. This induces an action of $W$ on $X^1$.
For $r\in W$ a conjugate of an element of $S$, the \emph{wall}
$\mathcal W_r$ of $r$ is the fixed point set of~$r$ in $X^1$. We call $r$ the \emph{reflection} in $\mathcal W_r$ (for fixed $\mathcal W_r$ such $r$
is unique). Each wall~$\mathcal W$ separates $X^1$ into two components, called \emph{half-spaces}, and a geodesic edge-path in~$X^1$ intersects~$\mathcal W$ at most once \cite[Lem~2.5]{Ronan_2009}. Consequently, the distance in~$X^1$ between $g,h\in W$ is the number of walls separating $g$ and $h$. 

For $g\in W$, let $\mathcal W(g)$ be the set of walls $\mathcal W$ in $X^1$ that separate $g$ from the identity element $\id\in W$ and such that
there is no wall $\mathcal W'$ separating $g$ from $\mathcal W$. 

We consider the partial order $\preceq$ on $W$, where $p\preceq g$ if $p$ lies on a
geodesic in $X^1$ from $\id$ to $g$. Equivalently, there is no wall separating $p$ from both $\id$ and $g$.

For $g\in W$, let $P(g)\subset W$ be the set of elements $p\in W$ satisfying $p\preceq g$ and such that there is no wall in $\mathcal W(g)$ separating $p$ from
$\id$. Note that $P(g)$ is nonempty, since $\id\in P(g)$. In Section~\ref{sec:well defined} we will prove the following.

\begin{thm}
\label{thm:well} For every Coxeter group $W,$ and each $g\in W,$ the set $P(g)$ contains a largest element with respect to $\preceq$.
\end{thm}

This largest element will be called \emph{the voracious projection} $p(g)$ of $g$. Note that $p(g)\neq g$ for $g\neq \id$.

We define the \emph{voracious language} $\mathcal V\subset S^*$ for $W$ inductively in the following way. Let $v\in S^*$ be a word of length $n$. If
$v$ represents the identity element of~$W$, then $v\in \mathcal V$ if and only if $v$ is the empty word. Otherwise, let $g\in W$ be the group element
represented by $v$, let $p$ be the voracious projection of $g$, and let $w=p^{-1}g\in W$ and $k=\ell(w)$. 
We declare $v\in \mathcal V$ if and only if
$v(n-k)\in \mathcal V$ and $v(n-k+1,n)$ represents~$w$. In particular, $v(n-k)$ represents $p$. It follows inductively that $n=\ell(g)$. Such a
language is called \emph{geodesic}. Note that the voracious language satisfies part~(i) of the definition of biautomaticity, and the set of words in $\mathcal V$ representing each $g\in G$ is finite.

The paths in $W$ formed by the words in the voracious language are inspired by the normal cube paths for $\mathrm{CAT}(0)$ cube complexes \cite[\S3]{Niblo-Reeves_1998} used to prove the biautomaticity for right-angled (or, more generally, cocompactly cubulated) Coxeter groups \cite{Niblo_Reeves_2003}.
Namely, the voracious projection $p(g)$ of $g$ is `so' voracious that the geodesics from $g$ to $p(g)$ intersect all the walls in $\W(g)$ (even if it means intersecting simultaneously other walls).

We will prove the Main Theorem with $\mathcal L$ the voracious language $\mathcal V$. 
It is clear from the definition that the voracious language is preserved by
the automorphisms of~$W$ stabilising~$S$, allowing us to apply \cite[Thm~6.7]{Swiatkowski_2006} on geodesic languages to obtain the Main Corollary concerning buildings.

\section{Voracious projection is well defined}
\label{sec:well defined}

\begin{defin} Let $r,q\in W$ be reflections. Distinct walls $\W_r,\W_{q}$ \emph{intersect}, if $\W_r$ is not contained in a half-space
for $\W_{q}$ (this relation is symmetric). Equivalently, $\langle r,q\rangle$ is a finite group. We say that such $r,q$ are
\emph{sharp-angled}, if $r$ and $q$ do not commute and $\{r,q\}$ is conjugate into $S$. In particular, there is a component of $X^1\setminus (\W_r\cup
\W_q)$ whose intersection $F$ with $X^0$ is a fundamental domain for the action of $\langle r,q\rangle$ on~$X^0$. We call such $F$ a \emph{geometric
fundamental domain for $\langle r,q\rangle$}.
\end{defin}

\begin{lemma}
\label{lem:key} Suppose that reflections $r,q\in W$ are sharp-angled, and that $g\in W$ lies in a geometric fundamental domain for $\langle
r,q\rangle$. Assume that there is a wall $\U$ separating $g$ from $\W_r$ or from $\W_q$. Let $\W'$ be a wall distinct from $\W_r,\W_q$ that is the
translate of $\W_r$ or $\W_q$ under an element of $\langle r,q\rangle$. Then there is a wall $\U'$ separating $g$ from $\W'$.
\end{lemma}

\begin{proof}
Consider the group $W_0< W$ generated by the $3$ reflections: $r,q$, and the reflection in $\U$. By \cite{Dyer_1990}, with the rank bound established in his Corollary~3.11
(see also \cite{Deodhar_1989} or \cite[Prop 3]{Tits_1988}), we have that $W_0$ can be identified with a Coxeter group of rank~3 such that
\begin{itemize}
\item
the reflections of $W_0$ are reflections of $W$, and
\item
the connected components of the complement  in $X^1$ of the walls of $W_0$ correspond (equivariantly) to the elements of $W_0$, and
\item
pairs of such components with intersecting closure in $X^1$ correspond to pairs of elements of $W_0$ differing by a generator of $W_0$.
\end{itemize}
Let $g_0$ be the element of $W_0$ corresponding to the above component containing $g$. Then $r,q$ are still sharp-angled in $W_0$, with $g_0$
in a geometric fundamental domain for $\langle r,q\rangle$. Thus to prove the lemma, it suffices to prove it for $W$ of rank~3.

We can assume $S=\{r,q,s\}$, where $\id$ lies in the same geometric fundamental domain for $\langle r,q\rangle$ as $g$. Since $\U$ is disjoint from $\W_r$ or $\W_q$, the group~$W$ is infinite, so we can assume without loss of generality $m_{sr}\geq 3$. If $m_{sq}\geq 3$, or $m_{sq}=2$ and
\begin{itemize}
\item $m_{sr}=\infty$, or
\item $m_{sr}\geq 4$ and $\W'\neq q\W_r$, or
\item $m_{sr}=3$ and $\W'\notin\{ q\W_r,r\W_q,qr\W_q\}$,
\end{itemize}
then $\W_s$ is disjoint from $\W'$. Since $g$ is separated from~$\W_r$ or $\W_q$ by $\U$, we have $g\neq \id$. Thus $\U'=\W_s$ separates $g$ from
$\W'$, as desired. See Figure~\ref{fig:single}(a).

\begin{figure}[h!]
	\begin{center}
		\includegraphics[scale=0.77]{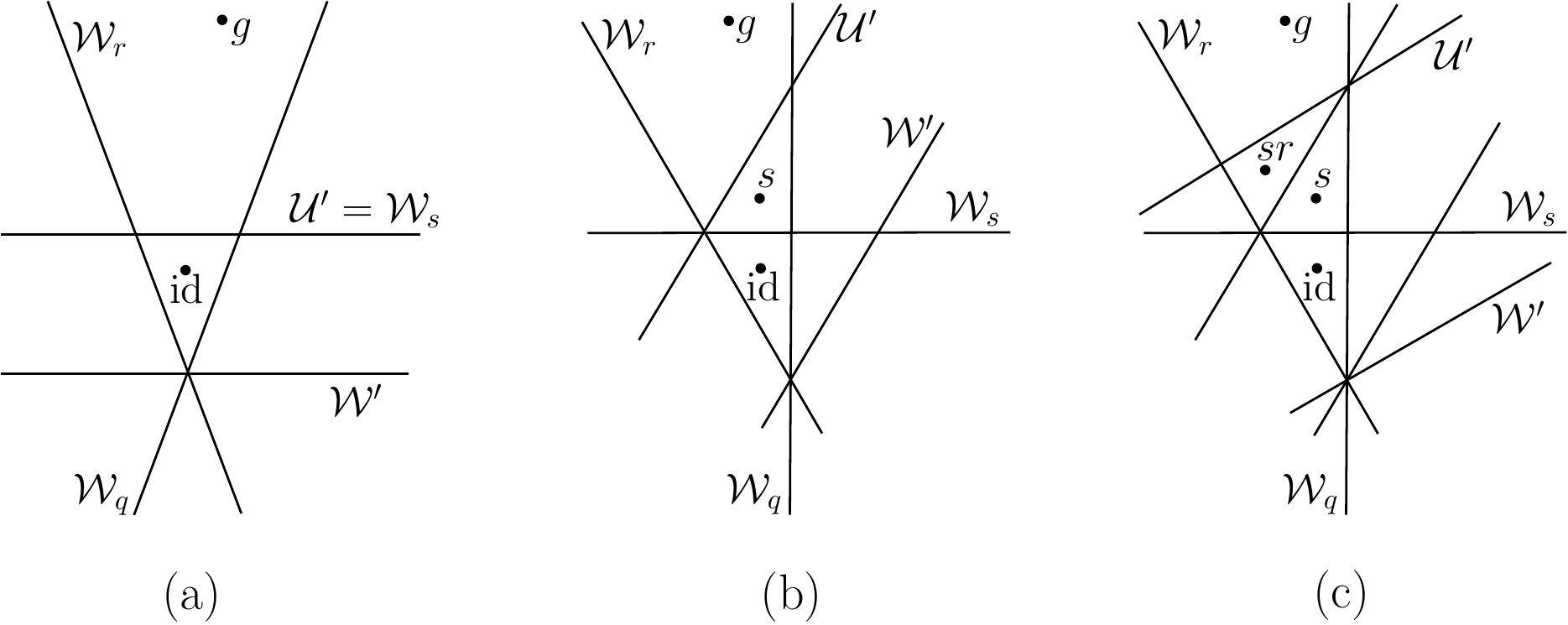}
	\end{center}
	\caption{Proof of Lemma~\ref{lem:key}.}
	\label{fig:single}
\end{figure}

If $m_{sq} =2, m_{sr}<\infty,$ and $\W'=q\W_r$, then let $\U'=s\W_r$. Note that $\U'$ is disjoint from $\W'$ since they are related by the point
symmetry $sq$. Furthermore, since $g$ is separated from~$\W_r$ or $\W_q$ by $\U$, and $m_{sr}<\infty,$ we have $g\neq \id, s$. Thus $\U'$ separates
$g$ from $\W'$, see Figure~\ref{fig:single}(b).

If $m_{sq} =2, m_{sr}=3,$ and $\W'=qr\W_q$, then let $\U'=sr\W_q$. Again $\U'$ is disjoint from $\W'$ since they are related by the point symmetry
$sq$. Furthermore, since $g$ is separated from~$\W_r$ or $\W_q$ by $\U$, we have $g\neq \id, s,sr$. Thus $\U'$ separates $g$ from $\W'$, see
Figure~\ref{fig:single}(c).

\begin{figure}[h!]
	\begin{center}
		\includegraphics[scale=0.7]{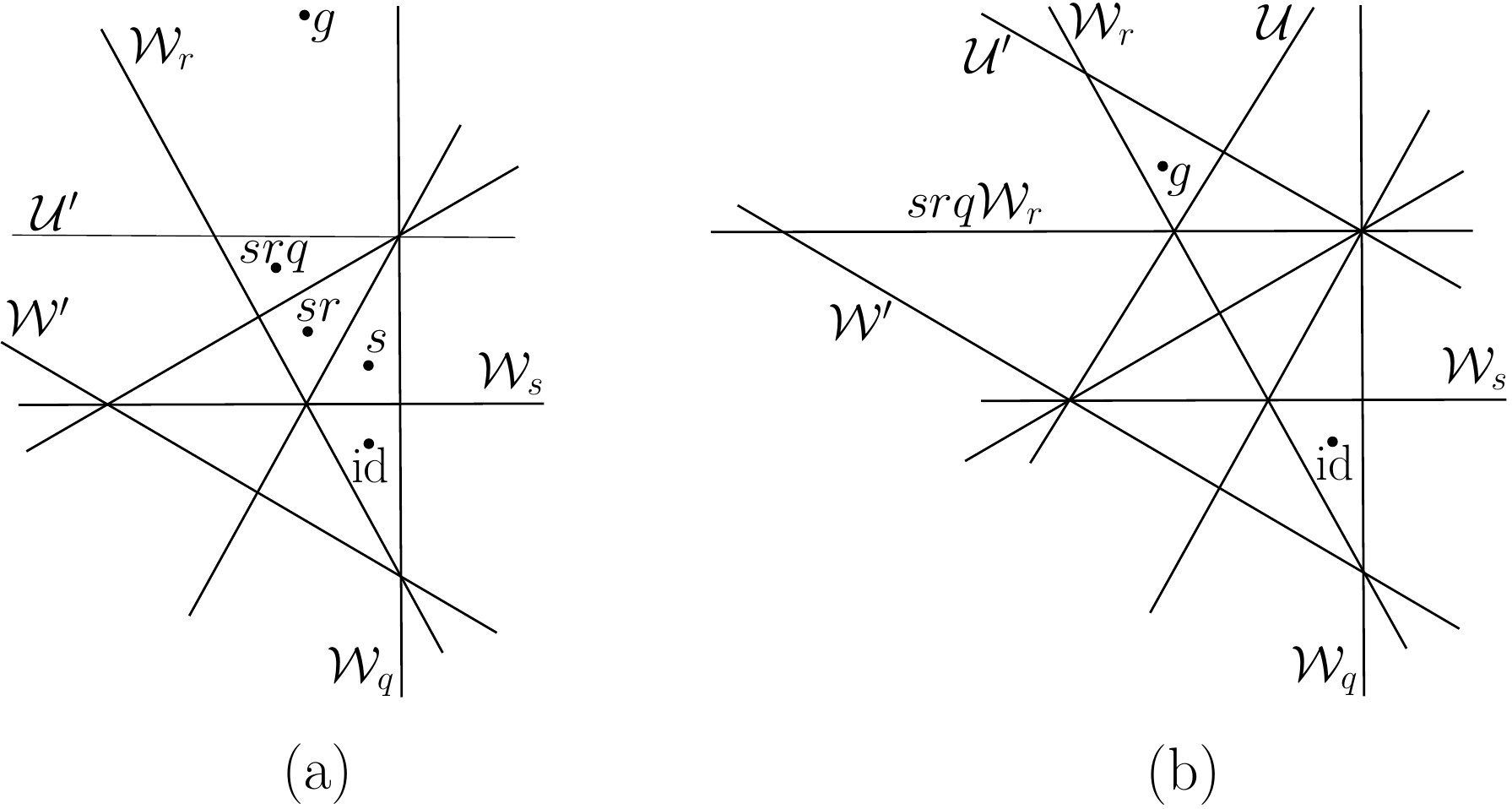}
	\end{center}
	\caption{Proof of Lemma~\ref{lem:key}, the case $m_{sq} =2, m_{sr}=3,$ and $\W'=r\W_q$.}
	\label{fig:last}
\end{figure}

It remains to consider the case where $m_{sq} =2, m_{sr}=3,$ and $\W'=r\W_q$. Suppose first the that wall $srq\W_r$ is disjoint from $\W'$, see
Figure~\ref{fig:last}(a). Then we can set $\U'=srq\W_r$, since $g\neq \id,s,sr,srq$. Second, suppose that $srq\W_r$ intersects $\W'$. Then let
$\U'=srqr\W_q$, see Figure~\ref{fig:last}(b). Note that $\U'$ is disjoint from $\W'$ since they are related by the point symmetry that is the composition of the reflections in 
$sr\W_q$ and in $\W_r$. If $\U'$ does not separate $g$ from $\W'$, then $g\in \{\id,s,sr,srq,srqr,srqrs\}$. Since $g$ is separated from~$\W_r$ or~$\W_q$ by $\U$, we
have that $g=srqrs$ is separated from $\W_q$ by  $\U=srqr\W_s$. But the reflection $r$ maps $\W_q$ and  $\U$ to $\W'$ and $srq\W_r$, contradicting
the assumption that the latter walls intersect.
\end{proof}

\begin{proof}[Proof of Theorem~\ref{thm:well}]
It suffices to show that for each $p_0,p_n\in P(g)$ there is $p\in P(g)$ satisfying $p_0\preceq p\succeq p_n$. Let $(p_0,p_1,\ldots,p_n)$ be the
vertices of a geodesic edge-path $\pi$ in $X^1$ from $p_0$ to $p_n$. Note that $\pi$ does not intersect the walls in $\W(g)$, since $p_0,p_n$ lie both in their
half-spaces containing $\id$. Furthermore, for any wall $\W$ containing $\id, g$ in the same half-space, we have that $p_0,p_n$ lie in that half-space,
and so do all~$p_i$. Consequently, $p_i\preceq g$, and so $p_i\in P(g)$.

We will now modify $\pi$ and replace it by another embedded edge-path (possibly not geodesic) from $p_0$ to $p_n$ with vertices in $P(g)$, so that there is no
$p_i$ with $p_{i-1}\succeq p_i\preceq p_{i+1}$. Then we will be able to choose $p$ to be the largest $p_i$ with respect to $\preceq$.

If $p_{i-1}\succeq p_i\preceq p_{i+1}$, then let $\W_r,\W_q$ be the walls separating $p_i$ from $p_{i-1},p_{i+1}$, respectively. Since
$p_{i-1}\preceq g,p_{i+1}\preceq g$, the walls $\W_r,\W_q$ intersect. Moreover, if $r$ and~$q$ do not commute, then $r,q$ are sharp-angled, with $g$
in a geometric fundamental domain for $\langle r,q\rangle$. We claim that all the elements of $R=\langle r,q\rangle (p_i)$ lie in~$P(g)$.

By \cite[Thm~2.9]{Ronan_2009}, we have that all the elements of $R$ lie on geodesics from~$p_i$ to $g$, and hence they are $\preceq g$. Since
$p_{i-1},p_{i+1}$ are both in $P(g)$, we have that $\W_r,\W_q\notin \W(g)$. It remains to justify that each remaining wall $\W'$ that is the translate
of $\W_r$ or $\W_q$ under an element of $\langle r,q\rangle$ does not belong to $\W(g)$. We can thus assume that $r$ and $q$ do not commute, since otherwise there is no
such remaining~$\W'$. Since $\W_r\notin \W(g)$, there is a wall~$\U$ separating $g$ from $\W_r$. By Lemma~\ref{lem:key}, there is a wall $\U'$ separating $g$
from~$\W'$, justifying the claim.

We now replace the subpath $(p_{i-1},p_i,p_{i+1})$ of $\pi$ by the second edge-path with vertices in~$R$ from $p_{i-1}$ to~$p_{i+1}$. This decreases the complexity
of $\pi$ defined as the tuple $(n_1,n_2,\ldots, n_{\ell(g)})$, where $n_j$ is the number of $p_i$ in $\pi$ with $\ell(p_i)=j$, with lexicographic
order. After possibly removing a subpath, we can assume that the new edge-path is embedded. After finitely many such modifications, we obtain the desired path.
\end{proof}

\section{Bounding the voracious projection}
\label{sec:bound}

\begin{prop}
\label{prop:bound} Let $W$ be a Coxeter group. There exists a constant $C=C(W)$ such that for each $g\in W,$ we have $\ell
\big(p(g)^{-1}g\big)\leq C$, where $p(g)$ is the voracious projection of~$g$.
\end{prop}

In the proof we need the following Parallel Wall Theorem.

\begin{thm}[{\cite[Thm~2.8]{Brink-Howlett_1993}}]
\label{thm:parallel} Let $W$ be a Coxeter group. There exists a constant $Q=Q(W)$ such that for each $g\in W$ and a wall $\mathcal W$ at distance $>
Q$ from $g$ in~$X^1$, there is a wall $\mathcal W'$ separating $g$ from $\mathcal W$.
\end{thm}

In particular, for $g\in W$, each of the walls in $\mathcal W(g)$ is at distance $\leq Q$ from $g$.

\begin{lemma}
\label{lem:bound} Let $W$ be a Coxeter group. There exists a constant $C_0=C_0(W)$ such that for each $g\in W$ and each $\W \in \W(g)$, there is
$h\in W$ satisfying $h\preceq g$, at distance $\leq C_0$ from $g$ in $X^1$, and separated from $g$ by $\W$.
\end{lemma}

\begin{proof} For $g\in W,$ and $\W \in \W(g)$, among elements $h\in W$ satisfying $h\preceq g$, and separated from $g$ by $\W$,
consider $h$ with minimal distance $C_0$ from $g$ in $X^1$. Our goal is to bound $C_0$ uniformly in $g$ and $\W$. We can assume $h=\id$.
Furthermore, by the minimality assumption, we have 
$\W=\W_s$ for $s\in S$.

Note that for $t\in S\setminus \{s\}$, the wall $\W_t$ does not separate $\id$ from $g$, since otherwise we could replace $\id$ by $t$ contradicting the minimality assumption.

We now use the \emph{Davis complex} $X$ of $W,$ which is obtained from $X^1$ by adding Euclidean polyhedra of edge length one spanned on all the cosets of finite $\langle T \rangle$ for $T\subseteq S$ (see \cite[Prop~7.3.4]{Davis_2008}). By \cite{Moussong_1988} (see also \cite{Bowditch_1995}), we have that $X$ is CAT(0). The fixed point sets of the
reflections of $W$ in $X$ are still called \emph{walls}, and they still separate $X$.

Let $\alpha$ be the minimal angle that can be formed between an intersection $M$ of a wall with one such polyhedron $\sigma$, and a geodesic in
$\sigma$ joining a vertex of $\sigma$ (all of which lie outside $M$) to a point of~$M$. Let $\gamma$ be the CAT(0) geodesic in $X$ between $\id$ and $g$. Since $X$ and $X^1$ are
quasi-isometric, we need to find a uniform bound for the length $c$ of $\gamma$. By Theorem~\ref{thm:parallel}, there is a uniform bound $d$ for the
CAT(0) distance from $g$ to $\W_s$ in $X$. Let $\sigma$ be the first polyhedron of $X$ with interior intersected by~$\gamma$, and let $M=\sigma\cap \W_s$. Note that $\gamma$
intersects $M$ (transversally at a point $m$) since $\gamma$ is disjoint from all $\W_t$, for $t\in S\setminus \{s\}$. The length $c_2$ of the second
component of $\gamma\setminus \{m\}$ is $\geq c-\mathrm{diam}(\sigma)$.

Since $X$ is CAT(0), by \cite[II.1.7(5)]{Bridson-Hae_1999}, we have $c_2\sin \alpha\leq d$, and so $c\leq \mathrm{diam}(\sigma)+d/\sin \alpha$, as desired.
\end{proof}

\begin{proof}[Proof of Proposition~\ref{prop:bound}]
By Theorem~\ref{thm:parallel}, there exists a constant $N=N(W)$ such that each $\W(g)$ has size $\leq N$. For each $g\in W$, applying at most $N$
times Lemma~\ref{lem:bound}, there is $h\in W$ satisfying $h\preceq g$, at distance $\leq C=C_0N$ from $g$, and separated from~$g$ by all $\W\in
\W(g)$. Consequently, $h\in P(g)$ and thus $h\preceq p(g)\preceq g$, implying $\ell \big(p(g)^{-1}g\big)\leq C$, as desired.
\end{proof}

\section{Fellow traveller property}
\label{sec:fellow}

In this section we verify parts (ii) and (iii) of the definition of biautomaticity.

\begin{lemma}
\label{lem:part2} Suppose that for $g,g'\in W$, we have $p(g)\preceq g'\preceq g$. Then $p(g')\preceq p(g)$.
\end{lemma}
\begin{proof} It suffices to prove $p(g')\in P(g)$. We have $p(g')\preceq g'\preceq g$. If $\W\in \W(g)$, and $\W$ separates $g'$ from $\id$, then we
 have $\W\in \W(g')$. Consequently, $\W$ separates $p(g')$ from $g'$, and hence from $g$.
\end{proof}

\begin{cor}
\label{cor:part2} The voracious language satisfies part (ii) of the definition of biautomaticity with $C$ replaced by $2C$ from
Proposition~\ref{prop:bound}.
\end{cor}
\begin{proof}
Let $g\in W$ and let $s\in S$ with $\ell(gs)<\ell(g)$. Let $g'=gs$. Since $p(g)\preceq g'\preceq g$, iterating Lemma~\ref{lem:part2}, we obtain $$\cdots \preceq p^2(g') \preceq p^2(g)\preceq p(g') \preceq p(g) \preceq
g'\preceq g,$$ 
where $p^k$ is defined inductively as $p^0(g)=g$ and $p^k(g)=p(p^{k-1}(g))$ for $k>0$.

Let $v,v'\in \mathcal V$ represent $g,gs,$ respectively. 
Let $1\leq i\leq \ell(g)$, and let $h,h'\in W$
be the elements represented by $v(i), v'(i)$, respectively.
We then have $\ell (p^{k}(g'))\leq i\leq \ell(p^{k}(g))$, or $\ell (p^{k+1}(g))\leq i\leq \ell(p^{k}(g'))$, for some $k\geq 0$. 
Furthermore, by Proposition~\ref{prop:bound}, we have that both $h,h'$ are at distance $\leq C$ from
$p^{k+1}(g)$ (respectively, $p^{k+1}(g')$) in $X^1,$ and so $\ell (h^{-1}h')\leq 2C$.
\end{proof}

\begin{lemma}
\label{lem:part3}
The voracious language satisfies part (iii) of the definition of biautomaticity.
\end{lemma}

\begin{proof}[Proof of Lemma~\ref{lem:part3}] 
Let $C'=2C(C+2Q)+2Q$, where $Q$ is the constant from Theorem~\ref{thm:parallel} and $C$ is the constant from Proposition~\ref{prop:bound}.

We prove part~(iii) of the definition of biautomaticity, with constant $C'$, inductively on $\ell(g)$, where we assume without loss of generality
$\ell(sg)>\ell(g)$. If $g=\id$, then there is nothing to prove. Suppose now $g\neq \id$. Let $v,v'\in \mathcal V$ represent $g,sg,$ respectively.

Assume first $\W_s\notin \W(sg)$. Then we have $\W(sg)=s\W(g)$. Consequently, $p(sg)=sp(g)$. In particular, 
the
words $v'\big(\ell(p(sg))\big)$ and $sv\big(\ell(p(g))\big)$ represent the same element $sp(g)$ of $W$. Then part (iii) of the definition of biautomaticity for $g$
follows inductively from part~(iii) for $p(g)$, for $i<\ell(p(sg))$, or from the definition of~$C$, for $i\geq\ell(p(sg))$.

Second, assume $\W_s\in \W(sg)$. Then $p(sg)$ and $\id$ lie in the same half-space for~$\W_s$. We claim that for any element $h\preceq p(sg)$, there
is no wall $\W'$ separating $h$ from $\W_s$. Indeed, otherwise a geodesic from $sg$ to $\id$ passing through~$h$ would intersect $\W'$ twice. By the
claim and Theorem~\ref{thm:parallel}, we have that $h$ is at distance $\leq 2Q$ from $sh$, which holds in particular for $h=p(sg)$. By the triangle
inequality, $sg$ and $sp(sg)$ are at distance $\leq C+2Q$ in $X^1$, and hence so are $g$ and $p(sg)$. By Corollary~\ref{cor:part2}, for $i< \ell(p(sg))$, we
have that the elements of $W$ represented by $v(i),v'(i)$ are at distance $\leq 2C(C+2Q)$ in $X^1$. Setting $h$ above to be the element of $W$ represented by $v'(i)$, we obtain by the triangle inequality
$\ell\big(v(i)^{-1}sv'(i)\big)\leq 2C(C+2Q)+2Q$, as desired. For $i\geq\ell(p(sg)),$ we have obviously $\ell\big(v(i)^{-1}sv'(i)\big)\leq 2C$ as well.
\end{proof}

\section{Regularity}
\label{sec:regularity}

A \emph{finite state automaton over $S$} (or, shortly, \emph{FSA}) is a finite directed graph $\Gamma$ with:
\begin{itemize}
\item vertex set~$A$, edge set $E\subseteq
A\times A$, 
\item
an edge labeling $\phi\colon E\to \mathcal P(S^*)$ (the power set of $S^*$), where each $\phi(e)$ is finite,
\item
a \emph{start state} $a_0\in A$, and 
\item
a distinguished set of
\emph{accept states} $A_\infty\subseteq A$. 
\end{itemize}
A word $v\in S^*$ is \emph{accepted by $\Gamma$} if there exists a decomposition $v=v_0\cdots v_m$ into
subwords, and a directed edge-path $e_0\cdots e_m$ in $\Gamma$ such that $e_0$ has initial vertex $a_0$, $e_m$ has terminal vertex in~$A_\infty$, and
$v_i\in \phi(e_i)$ for each $i=0,\ldots,m$. A subset of $S^*$ is a \emph{regular language} if it is the set of accepted words for some FSA over $S$.

\begin{prop}
\label{prop:regular} The voracious language is regular.
\end{prop}

To prove Proposition~\ref{prop:regular}, we define an FSA $\Gamma$ over $S$ that will accept exactly the voracious language.

\begin{defin}
\label{def:reg} Let $\U$ be the set of walls that are not separated from $\id$ by any other wall, which is finite by Theorem~\ref{thm:parallel}. The
vertex set $A$ of our FSA $\Gamma$ is the power set $\mathcal P(\U)$. Let $a_0=\emptyset\subset \U$, and $A_\infty=A$.

To define the edges from $a\in A$, suppose that $w\in W$ satisfies:
\begin{enumerate}[(a)]
\item $w$ is not separated from $\id$ by any wall in $a$, and
\item $w$ is separated from each wall in $a$ by another wall, and
\item $p(w)=\id$.
\end{enumerate}
We then put en edge $e$ in $E$ between $a$ and $w^{-1}\W(w)$ with $\phi(e)$ consisting of all the minimal length words representing $w$.
\end{defin}

\begin{proof}[Proof of Proposition~\ref{prop:regular}]
Let $\Gamma$ be the FSA from Definition~\ref{def:reg}, and let $\mathcal V$ be the voracious language.
We argue inductively on $j\geq 0$ that, among the words $v\in S^*$ of length $\leq j$,
\begin{itemize}
\item
$\Gamma$ accepts exactly the words in $\mathcal V$, and
\item
the accept state of each such word $v$ is $g^{-1}\W(g)$, where $v$ represents $g\in W.$
\end{itemize}
This is true for $j=0$ by our choice of $a_0$. Now let $n>0$ and suppose that we have verified the inductive hypothesis for all $j<n$. Let $v$ be a
word in $S^*$ of length $n$.

Suppose first that $v$ is a word in $\mathcal V$ representing $g\in W$. Let $p=p(g), w=p^{-1}g.$ 
By the definition of~$\mathcal V$, we
have $v\big(\ell(p)\big)\in \mathcal V$. Moreover, $v\big(\ell(p)+1,n\big)$ represents~$w$. By the inductive hypothesis, $\Gamma$ accepts $v\big(\ell(p)\big)$. Furthermore, $v\big(\ell(p)\big)$
labels some directed edge-path in $\Gamma$ from $a_0$ to $p^{-1}\W(p)$. We will now show that $\Gamma$ has an edge~$e$ from $a=p^{-1}\W(p)$ to
$g^{-1}\W(g)$, with $\phi(e)$ consisting of all minimal length words representing $w$. To do that, we verify the conditions for $w$ from
Definition~\ref{def:reg}. Condition~(a) follows from the fact that $p\preceq g$ and so $g$ is not separated from $p$ by any wall in $\W(p)$. Since
$p\in P(g)$, we have that $\W(p)$ is disjoint from $\W(g)$, which implies condition~(b) and $g^{-1}\W(g)=w^{-1}\W(w)$. Consequently, $p(g)p(w)\in
P(g)$ implying $p(w)=\id$, which is condition~(c).

Conversely, let $v$ be accepted by $\Gamma$ and suppose that $v=v_0\cdots v_m$ as in the definition of an accepted word. By the inductive hypothesis,
the word $v_0\cdots v_{m-1}$ belongs to $\mathcal V$ and represents $p\in W$ such that $e_m$ starts at $a=p^{-1}\W(p)$. By the definition of the
edges, $v_{m}$ is a minimal length word representing an element~$w$ satisfying the conditions~(a,b,c). By condition~(a), we have $p\preceq g$ for
$g=pw$. By condition~(b), we have $w^{-1}\W(w)=g^{-1}\W(g)$. Thus by condition~(c), we have $p=p(g)$. Consequently, $v\in \mathcal V$, as desired.
\end{proof}

\bibliography{mybib}{}
\bibliographystyle{plain}

\end{document}